\documentclass{amsart}
\usepackage[utf8]{inputenc}
\usepackage{amsmath}
\usepackage{amssymb}
\usepackage{amsthm}
\usepackage{tkz-graph} 
\usepackage{enumerate} 
\usepackage{bm}
\usepackage{color}
\usepackage{pstricks}

\makeatletter
\newcommand{\addresseshere}{%
  \enddoc@text\let\enddoc@text\relax
}
\makeatother

\theoremstyle{definition}
\newtheorem{theorem}{Theorem}

\newtheorem{lemma}{Lemma}

\newtheorem{remark}{Remark}

\title{An improved method for recursively computing upper bounds for two-colour
Ramsey numbers}
\author{Oliver Kr\"uger}
\address{Department of Mathematics, Stockholm University, Sweden}
\email{okruger@math.su.se}
\date{\today}

\begin{document} 
\maketitle 

\begin{abstract}
The two-colour Ramsey number $R(m,n)$ is the least natural number $p$
such that any graph of order $p$ must contain either a clique of size $m$
or an independent set of size $n$.
We exhibit a method for computing upper bounds for $R(m,n)$ recursively,
using known upper bounds of $R(\cdot,\cdot)$ with lower values for at
least one of the arguments.
We also give an example of how this method could be used to improve
several of the best known bounds that are available in the literature
(which however soon will be obsolete due to a forthcoming work).  
\end{abstract}

\section{Introduction}

We say that a graph $G$ is a $(m,n)$-graph if it contains no clique of size $m$ and
no independent set of size $n$. If the order of $G$ is $p$, then we say that $G$ is
a $(m,n;p)$-graph. The two-colour Ramsey number $R(m,n)$ is defined to be the least
natural number $p$ such that there are no $(m,n;p)$-graphs. A list of bounds on these
numbers are maintained in a dynamic survey authored by Radziszowski in \cite{dynsurv}.

We present a new method for computing upper bounds on Ramsey numbers $R(m,n)$
recursively from known upper bounds of $R(m_0,n_0)$, where $m_0 \leq m$ and
$n_0 \leq n$ with at least one of the inequalities being strict. We will also
show how this method can be effectively used to improve the upper bounds on
several of the upper bounds of $R(m,n)$ listed in \cite{dynsurv}. These
improvements will soon however be obsolete due to a forthcoming work by
Angeltveit and McKay (see Remark \ref{remark:obsolete}).  

The new method presented in this paper is an enhancement of the method of \cite{hwplus}
to derive bounds on the minimum edge numbers of $(m,n;p)$-graphs. These may then be used to obtain
stronger results on $(m,n+1)$- and $(m+1,n)$-graphs.

\section{The new method}

Let $e(m,n;p)$ and $E(m,n;p)$ denote the minimum and maximum number of
edges in a $(m,n;p)$-graph, respectively. We will denote the complement
of a graph $G$ by $\overline{G}$. The subgraph of $G$ induced
by the neighbours of a vertex $v$ is $G_v^+$, while the subgraph that is
induced by the vertices that are not adjacent to $v$ is $G_v^-$. $N(K_3;G)$
is the number of triangles in $G$ and $N(K_3;G,v)$ is the number of
triangles in $G$ that contain the vertex $v$. $n_d$ denotes the number
of vertices in $G$ that has degree $d$, and $d_v$ the degree of the
vertex $v$.

The methods used to prove the following two theorems are similar to those
used to prove Lemma \ref{hwplus}. In particular we use Goodman's
lemma for counting the total number of triangles in $G$ and
$\overline{G}$ (see \cite[Lemma 1]{goodman}), which states that if
$p$ is the order of $G$ then
$$N(K_3;G) + N(K_3;\overline{G}) = \binom{p}{3} - \frac{1}{2} \sum_v
d_v(p-d_v-1).$$

\begin{theorem}
\label{mymain}
Let $p,\alpha,\beta,\gamma,\delta$ be as in Lemma \ref{hwplus}. Then
$$(p-1)(p-2) \leq \max_{d \in [p-1-\delta,\gamma]}
2\binom{p-d-1}{2} + 2\Delta(m,n,p,d) + 3d(p-d-1),$$
where
$$\Delta(m,n,p,d) = E(m-1,n;d) - e(m,n-1;p-d-1).$$
\end{theorem}
\begin{proof}
Let $G$ be a $(m,n;p)$-graph. Clearly $d_v \in I := [p-1-\delta,\gamma]$ for all
vertices $v \in V(G)$, since $G_v^+$ is a $(m-1,n;d_v)$-graph and $G_v^+$ is
a $(m,n-1;p-1-d_v)$-graph. Note that $N(K_3;\overline{G},v) \leq
\binom{p-d_v-1}{2} - e(m,n-1;p-d_v-1)$ and $N(K_3;G,v) \leq E(m-1,n;d_v)$. Thus
we have
$$3N(K_3;\overline{G}) + 3N(K_3;G) \leq \sum_{d \in I}
\left( \binom{p-d-1}{2} + \Delta(m,n,p,d) \right)n_d.$$
By a straightforward application of Goodman's lemma we get
$$p(p-1)(p-2) \leq \sum_{d \in I} \left(
2 \binom{p-d-1}{2} + 2\Delta(m,n,p,d) + 3d(p-d-1)
\right)n_d,$$
and the lemma follows.
\end{proof}

The minimal edge numbers, $e(m,n;p)$, have been studied previously. For low
values of $m,n$ and $p$ the exact value of these are known. The following lemma will
however give us sufficiently good bounds for $e(m,n;p)$ (and
maximal edge numbers $E(m,n;p)$)
for use in Theorem \ref{mymain} to derive the new bounds listed in
Table \ref{t1} and \ref{t2}.

\begin{theorem}
\label{myedgebounds}
Let $p,\alpha,\beta,\gamma,\delta$ be as in Lemma \ref{hwplus}. Then
\begin{align*}
e(m,n;p) &\geq \max \left\{\frac{p(p-\delta-1)}{2}, \frac{A - \sqrt{A^2-B}}{12} \right\} \\
E(m,n;p) &\leq \min \left\{\frac{p \gamma}{2}, \frac{A + \sqrt{A^2-B}}{12} \right\},
\end{align*}
where
$$A = (\alpha - \beta + 3(p-1))p,\quad B = 12p^2(p-1)(p-\beta-2).$$
\end{theorem}
\begin{proof}
Let $G$ be a $(m,n;p)$-graph. The bounds $e(m,n;p) \geq p(p-\delta-1)/2$ and
$E(m,n;p) \leq p\gamma/2$ are clear since $d_v \in [p-1-\delta,\gamma]$ for
all $v \in V(G)$. Using the bounds $3N(K_3;G) \leq \sum_{v \in V(G)}
\frac{\alpha d_v}{2}$ and $3N(K_3;\overline{G}) \leq \sum_{v \in V(G)}
\frac{\beta (p-1-d_v)}{2}$ together with Goodman's theorem we get
$$6\binom{p}{3} - \beta p(p-1) \leq \sum_v \left(
(\alpha - \beta + 3(p-1))d_v - 3d_v^2
\right).$$
Therefore, by the handshaking lemma and the Cauchy-Schwartz inequality, we
have $6\binom{p}{3} - \beta p (p-1) \leq 2(\alpha - \beta + 3(p-1))e - 12e^2/p$,
which we can rewrite as $0 \leq -12e^2 + 2Ae - B/12$, which gives us that
$e \in [(A - \sqrt{A^2 - B})/12,(A + \sqrt{A^2-B})/12]$.
\end{proof} 

\section{Example: Application to small Ramsey numbers as given in the
literature}

We will now show how the method described in the previous section could
be applied to obtain improved upper bounds for some small Ramsey numbers.
Note, however, that although this i an improvement on the best upper
bounds in the literature these are not the best known bounds (see 
Remark \ref{remark:obsolete}).

We will also use some previously known methods to compute upper bounds
and indicate where in the example calculation these are sufficient.
The following recursive classical bound, which appears in \cite{gg}, will be
used as the most trivial way to obtain new bounds.

\begin{lemma}[Greenwood and Gleason \cite{gg}]
\label{sumbound}
$R(m,n) \leq R(m-1,n) + R(m,n-1)$, with strict inequality if both $R(m-1,n)$ and $R(m,n-1)$ are even.
\end{lemma}

The essence of \cite[Theorem 1]{hwplus}, in the context that we will use it,
is the following lemma.

\begin{lemma}[Huang, Wang, Sheng, Yang, Zhang and Huang \cite{hwplus}]
\label{hwplus}
If $p \leq R(m,n) - 1$, $\alpha \geq R(m-2,n)-1$, $\beta \geq R(m,n-2)-1$,
$\gamma \geq R(m-1,n)-1$ and $\delta \geq R(m,n-1)-1$, then
$$(p-1)(p-2-\alpha) \leq \max_{d \in [p-1-\delta,\gamma]} (-3d^2 + (\alpha - \beta + 3(p-1))d + (\beta - \alpha)(p-1)).$$
\end{lemma}

A selection of the upper bounds we can obtain with the methods presented in
this paper
have been summarised in the following tables. Boldface indicates
that the bound is improved compared to the one in \cite{dynsurv} and the superscript
letters indicate by what method the bound has been obtained.

\begin{remark}
\label{remark:obsolete}
Angeltveit and McKay \cite{angeltveitmckayprivate}
 have improved on many of the upper bounds that are
listed in \cite{dynsurv}, completely independently of this paper. Their
improvements are stronger and therefore the bounds listed in the tables
below are generally not the best known. Their work is however still in
progress, and has not yet been published.

The numbers actually listed in the below table may therefore be considered obsolete.
The method used to compute them is however interesting, and this this example
calculation is made to illustrate its potential.
\end{remark}

\begin{table}[h]
\tiny{
\begin{tabular}{|l r||c|c|c|c|c|c|c|c|c|c|c|}
\hline
 & $n$ & 5 & 6 & 7 & 8 & 9 & 10 & 11 & 12 & 13 & 14 & 15 \\
$m$ &  &   &   &   &   &   &    &    &    &    &    &    \\
\hline
5 & & 48 & 87 & \textbf{142}$^b$ & \textbf{215}$^c$ & 316 & 442 & \textbf{629}$^c$ & \textbf{846}$^c$ & \textbf{1102}$^c$ & \textbf{1442}$^c$ & \textbf{1832}$^c$ \\
6 & & & 165 & 298 & 495 & 780 & 1171 & \textbf{1782}$^c$ & \textbf{2549}$^c$ & \textbf{3526}$^c$ & \textbf{4927}$^c$ & \textbf{6614}$^c$ \\
7 & & & & \textbf{539} & \textbf{1029}$^b$ & \textbf{1711}$^c$ & \textbf{2775}$^c$ & \textbf{4518}$^c$ & \textbf{6821}$^c$ & \textbf{10017}$^c$ & \textbf{14841}$^c$ &  \textbf{20928}$^c$ \\
8 & & & & & \textbf{1865} & \textbf{3576}$^a$ & \textbf{6061}$^b$ & \textbf{10297}$^c$ & \textbf{16777}$^b$ & \textbf{25933}$^c$ & \textbf{40140}$^c$ &  \textbf{59916}$^c$ \\
9 & & & & & & \textbf{6582} & \textbf{12643}$^a$ & \textbf{22161}$^b$ & \textbf{38000}$^c$ & \textbf{62763}$^c$ & \textbf{100614}$^c$ & \textbf{157549}$^c$ \\
10  & & & & & & & \textbf{23327} & \textbf{45488}$^a$ & \textbf{80231}$^b$ & \textbf{139767}$^c$ & \textbf{236772}$^b$ & \textbf{385139}$^c$ \\
\hline 
\end{tabular}
}
\caption{Upper bounds of $R(m,n)$. Boldface indicates improved bounds (compared
to \cite{dynsurv}).
Superscript is $a$ if Lemma \ref{sumbound} is sufficient, $b$ if Lemma \ref{hwplus} is sufficient but $c$
if this bound requires the new methods from Theorems \ref{mymain} and
\ref{myedgebounds}. Diagonal entries improved by using \cite[Theorem 5.1]{shi}.
Also see Remark \ref{remark:obsolete}.}
\label{t1}
\end{table}

\begin{table}[h]
\tiny{
\begin{tabular}{|l r||c|c|c|c|c|c|c|c|}
\hline
 & $n$ & 16 & 17 & 18 & 19 & 20 & 21 & 22 & 23 \\
$m$ &  &    &    &    &    &    &    &    &    \\
\hline
4 & & 514    &615$^b$   & 720$^c$ & 851$^a$ & 988$^b$ & 1129$^c$ & 1300$^a$ & 1476$^b$ \\
5 & & 2321$^c$ & 2916$^c$ & 3576$^c$ & 4397$^c$ & 5350$^c$ & 6381$^c$ & 7651$^c$ & 9074$^c$ \\
6 & & 8745$^c$ & 11596$^c$ & 14903$^c$ & 19037$^c$ & 24272$^c$ & 30177$^c$ & 37497$^c$ & 46374$^c$ \\
\hline
\end{tabular}
}
\caption{Upper bounds of $R(m,n)$. Same conventions as for \mbox{Table
\ref{t1}}.} 
\label{t2}
\end{table} 

\bibliographystyle{plain}


\addresseshere

\end{document}